\documentclass[12pt,reqno]{amsart}
\usepackage{amssymb,amsmath,amsthm,enumerate,bbm}
\usepackage[a4paper]{geometry}
\DeclareMathOperator{\sign}{sign}

\DeclareMathOperator{\rank}{rank}
\DeclareMathOperator{\supp}{supp}
\DeclareMathOperator{\dist}{dist}

\DeclareMathOperator{\VMO}{VMO}
\DeclareMathOperator{\BMO}{BMO}

\newcommand{\abs}[1]{\lvert#1\rvert}
\newcommand{\Abs}[1]{\left\lvert#1\right\rvert}
\newcommand{\norm}[1]{\lVert#1\rVert}

\newcommand{\bbT}{{\mathbb T}}
\newcommand{\bbR}{{\mathbb R}}
\newcommand{\bbC}{{\mathbb C}}

\newcommand{\bbZ}{{\mathbb Z}}

\newcommand{\bbD}{{\mathbb D}}

\newcommand{\wh}{\widehat}

\newcommand{\ov}{\overline}

\newcommand{\calP}{\mathcal{P}}
\newcommand{\calB}{\mathcal{B}}

\newcommand{\calR}{\mathcal{R}}

\newcommand{\Sch}{\mathbf{S}}

\numberwithin{equation}{section}

\theoremstyle{plain}
\newtheorem{theorem}{\bf Theorem}[section]
\newtheorem*{theorem*}{Theorem 1.1$'$}
\newtheorem{lemma}[theorem]{\bf Lemma}
\newtheorem{proposition}[theorem]{\bf Proposition}
\newtheorem{assumption}[theorem]{\bf Assumption}

\theoremstyle{definition}

\theoremstyle{remark}
\newtheorem*{remark*}{\bf Remark}
\newtheorem{remark}[theorem]{\bf Remark}

\newcommand{\wt}{\widetilde}

\newcommand{\const}{\mathrm{const}}
\newcommand{\1}{\mathbbm{1}}

\begin{document}

\title[Best rational approximation]{Best rational approximation of  functions with logarithmic singularities}

\author{Alexander Pushnitski}
\address{Department of Mathematics, King's College London, Strand, London, WC2R~2LS, U.K.}
\email{alexander.pushnitski@kcl.ac.uk}

\author{Dmitri Yafaev}
\address{Department of Mathematics, University of Rennes-1,
Campus Beaulieu, 35042, Rennes, France}
\email{yafaev@univ-rennes1.fr}

\subjclass[2000]{41A20, 47B06, 47B35}

\keywords{Hankel operators,  asymptotics of singular values, logarithmic singularities, rational approximation}

\begin{abstract}
We consider functions $\omega$ on the unit circle $\bbT$ 
with a finite number of logarithmic singularities.
We study the  approximation of $\omega$ by rational functions and find an asymptotic formula for the distance in the $\BMO$-norm between $\omega$ and the set of 
rational functions of degree $n$ as $n\to\infty$. 
Our approach relies on the Adamyan-Arov-Krein theorem and on the
study of the asymptotic behaviour of singular values of Hankel operators. 
\end{abstract}


\maketitle

\section{Introduction}\label{sec.z}
 
\subsection{Overview}

The rate of convergence of both rational and polynomial  approximations 
to a given function $\varphi$ is determined by the  smoothness of $\varphi$. 
Of course in general the rational approximations converge much faster than the polynomial ones.

Let us briefly describe the fundamental results of approximation theory relevant to this paper; 
see \cite{Stahl} for more information. 
We denote by  $\calP_n$  the set of all polynomials in $x\in\bbR$ of degree $\leq n$. 
Similarly,  $ \mathcal{T}_n$  is the set of all trigonometric polynomials   of degree $\leq n$ defined on the unit circle $\bbT$. 
According to the classical Jackson-Bernstein theorem (see, e.g., the book \cite{dVL}), 
the distance between $\varphi$ and $\mathcal{T}_n$ in the $L^\infty$-norm satisfies the estimate
$$
\dist_{L^\infty(\bbT)}\{\varphi, \mathcal{T}_n\} =O (n^{-\alpha}),\quad \alpha >0, \quad n\to\infty,
$$
if and only if $\varphi$ belongs to the H\"older-Zygmund class $\Lambda_{\alpha}$
(the  definitions of relevant function classes are collected in the Appendix).

Further, for the function 
$\varphi(x)=\abs{x}^\alpha$ defined on some interval of the real line, for example on $[-1,1]$, 
S.~N.~Bernstein \cite{Bern1, Bern2}  proved 
the existence of the limit  
\begin{equation}
\lim_{n\to\infty}n^\alpha\dist_{L^\infty(-1,1) }\{\abs{x}^\alpha,\calP_n\} ={\sf b}(\alpha)
\label{eq:Bern}
\end{equation}
where ${\sf b}(\alpha)\neq 0$ if $\alpha\neq 2,4, \ldots$. The number ${\sf b}(\alpha)$ is known as the \emph{Bernstein constant}.

Next, consider the problem of rational approximation. 
The degree of a rational function $p/q$ ($p$, $q$ are polynomials with no non-constant common 
divisors) is defined as $\max\{\deg p,\deg q\}$. 
We denote by $\calR_n$ the set of all rational functions 
of degree $\leq n$ in the complex plane. 
D.~Newman proved in \cite{Newman} that for the function 
$\varphi(x)=\abs{x}$ on the interval $[-1,1]$, the distance between $\varphi$ 
and $\calR_n$ in the $L^\infty$ norm satisfies the estimates
$$
e^{-c_{1} \sqrt{n}}  
\leq 
\dist_{L^\infty(-1,1)}\{\varphi,\calR_n\}
\leq 
e^{-c_{2} \sqrt{n}} 
$$
with some positive constants $c_{1}$, $c_{2}$.
This result was extended by A.~A.~Gonchar in
\cite{Gonchar1, Gonchar} to the functions $\varphi(x)=|x|^\alpha$; 
he established the same estimate   for   all $\alpha>0$,
$\alpha\neq 2,4,\ldots$.  
More recently, H.~Stahl \cite{Stahl} proved a remarkable result: he showed
that for such functions  one has the asymptotic relation
\begin{equation}
\lim_{n\to\infty} e^{\pi\sqrt{\alpha n}}\dist_{L^\infty(-1,1)}\{\abs{x}^\alpha,\calR_n\}
=
4^{1+\alpha/2}\abs{\sin\tfrac12\pi \alpha}, \quad \alpha>0;
\label{eq:St}
\end{equation}
see \cite{Stahl} for the history of the problem.

 In this paper, we discuss the rational approximation of functions with logarithmic
singularities of the type $(-\ln\abs{x})^{-\alpha}$  near $x=0$.
More precisely, let us fix $\alpha>0$ and consider the function 
$$
\varphi_+(x)=
\begin{cases}
(-\ln x)^{-\alpha}& x\in (0, 1/2]
\\
0 & x\in[-1/2, 0]
\end{cases}
$$
on the interval $[-1/2,1/2]$. 
Clearly, $\varphi_+$ does not satisfy the H\"older continuity condition (with any exponent) 
and so, according to the Jackson-Bernstein theorem, $\dist_{L^\infty}\{\varphi_+,\calP_n\}$
goes to zero slower than any power of $n^{-1}$. 
On the other hand, 
A.~A.~Gonchar in \cite{Gonchar} proved the two-sided estimates
\begin{equation}
c n^{-\alpha}  
\leq 
\dist_{L^\infty(-1/2,1/2)}\{\varphi_{+},\calR_n\}
\leq
C  (\ln n/n)^\alpha
\label{eq:Go}
\end{equation}
with some positive constants $c$ and $C$. 

Our aim is to obtain an asymptotic relation for the function $\varphi_{+} (x)$ 
(and for more general functions with similar singularities) in the spirit of \eqref{eq:St} but with $n^\alpha$ instead of the exponential $e^{\pi\sqrt{\alpha n}}$.  
 We obtain such a  relation, but for the $\BMO$-norm  instead of the $L^\infty$-norm. 
In fact, in harmonic analysis the space  $\BMO$ 
(functions with bounded mean oscillation)
often plays the role of a proper substitute for $L^\infty$; 
this space is only slightly larger than $L^\infty$ and $L^\infty\subset \BMO\subset  L^p$ for any $p <\infty$. 
Further, this space is   particularly well adapted to treating functions with logarithmic singularities and allows us to study unbounded functions. 
Indeed, along with $\varphi_+$ we consider the   function
$$
\varphi_0(x)
=(-\ln\abs{x})^{1-\alpha}, \quad x\in[-1/2,1/2],
$$
which \emph{a priori} looks more singular than $\varphi_+$ because of the extra factor $\ln\abs{x}$.
For $0<\alpha<1$, this function is unbounded, but it is in the VMO class (functions with vanishing mean oscillation). 
Just as for $\varphi_+$,
we show that $n^\alpha\dist_{\BMO}\{ \varphi_0,\calR_n\}$  attains a finite positive limit as $n\to\infty$
which we compute explicitly. 
Comparing these facts with the classical Bernstein result \eqref{eq:Bern}, we see that, for functions with logarithmic singularities, rational approximations play the same role as polynomial approximations play for
functions with power singularities. 

It will be convenient for us to work with functions defined on the unit circle $\bbT$ of the 
complex plane, instead of on an interval of the real line. So, to be more precise, 
below we consider the analogues of $\varphi_+$ and $\varphi_0$ on the unit circle.

Our approach relies  on a combination of 
the fundamental  Adamyan-Arov-Krein (AAK) theorem \cite{AAK}
and of our previous results \cite{II,LOC} on the asymptotic behaviour 
of singular values of Hankel operators of a certain class.
The AAK theorem relates the rational approximation 
of a function $\varphi$ (defined on $\bbT$) in the $\BMO$-norm to the singular values 
of the Hankel operator with the symbol $\varphi$. 
This explains why we work with the $\BMO$-norm rather than with the $L^\infty$-norm. 

Using the AAK theorem, V.~V.~Peller \cite{Peller} has obtained 
the following analogue of the Jackson-Bernstein theorem for the rational approximations in the BMO norm.
He proved that 
\begin{equation}
\dist_{\BMO(\bbT)}\{\varphi,\calR_n\} =O (n^{-\alpha}),\quad \alpha >0, \quad n\to\infty,
\label{eq:BesPel}
\end{equation}
if and only if $\varphi$ belongs to a certain Besov-Lorentz class, denoted by 
$\mathfrak{B}^{\alpha}_{1/\alpha,\infty}$ in \cite{Peller}. 
We reproduce the definition of this class in the Appendix. 
Of course, the specific functions $\varphi$ with logarithmic singularities that we
consider in this paper belong to this class.

\subsection{BMO and VMO}
We denote by $\bbT$ the unit circle in the complex plane, equipped with the 
normalized Lebesgue measure $dm(\mu)= (2\pi i \mu)^{-1} d\mu$, $\mu\in\bbT$, 
and set $L^p\equiv L^p(\bbT)$. 
For $f\in L^1$, let
$$
\wh f(j)= \int_{\bbT}  f(\mu) \mu^{-j} dm(\mu),
\quad 
j\in\bbZ ,
$$
be the Fourier coefficients of $f$. 
For $1\leq p\leq\infty$, the Hardy classes $H^p_+$ and $H^p_-$ are defined 
in a standard way as
$$
H^p_+=\{f\in L^p(\bbT): \wh f(j)=0 \quad \forall j<0\},
\quad
H^p_-=\{f\in L^p(\bbT): \wh f(j)=0 \quad \forall j\geq0\}.
$$
We denote by $P_+: L^2 \to H^2_+ $ and $P_{-}: L^2 \to H^2_-$ the orthogonal projections onto $H^2_+ $ 
and $H^2_- $.
There is a lack of complete symmetry between $H^2_+$ and $H^2_-$ because the constant
functions belong to $H^2_+$ but not to $H^2_-$.
This results in a slight asymmetry in some of the formulas below. 

The class $\BMO(\bbT)=:\BMO$ can be described in many equivalent ways, with equivalent choices
for the norm;
see, e.g., \cite{Koosis}. 
Since we are interested in the \emph{asympotics} of the distance in the $\BMO$-norm, the precise
choice of the norm will be important for us.

Let us start by fixing the $\BMO$-norm for 
functions  analytic outside the unit disk. 
A function $f\in H^2_-$ belongs to $\BMO$ if and only if 
$f-g\in L^\infty$ for some $g\in H^2_+$; then we   set
\begin{equation}
\| f\|_{ \BMO } =\inf\{ \|f-g   \|_{L^\infty}: g\in H^2_{+}\}, \quad f\in H^2_-.
\label{eq:BMO}
\end{equation}

\begin{remark*} 
Let $H^1_+(0)=\{f\in H^1_+: \wh f(0)=0\}$. As is well known (see,  e.g. \cite[Section VII.A.1]{Koosis}),
the norm \eqref{eq:BMO} coincides with the norm of $f$ in the dual space $H^1_+(0)^*$.
\end{remark*}

Next, for an arbitrary $f\in L^2$ we set
\begin{equation}
\norm{f}_{\BMO}  
=  
\max\{\norm{P_{-} f}_{\BMO},\norm{P_{-}\bar{f}}_{\BMO},\abs{\wh{f}(0)}\},
\label{eq:BMBM}
\end{equation}
if the right hand side is finite. 
Here $\norm{P_-f}_{\BMO}$, $\norm{P_{-}\bar{f}}_{\BMO}$ 
are defined by \eqref{eq:BMO}.
Clearly, the norm  \eqref{eq:BMBM} is invariant with respect to the complex conjugation,
$$
\norm{\overline{f}}_{\BMO}=\norm{f}_{\BMO}.
$$
With this definition, we have
\begin{equation}
\norm{P_+ f}_{\BMO}
=
\max\{\norm{P_- \bar{f}}_{\BMO},\abs{\wh f(0)}\},
\label{z7}
\end{equation}
\begin{equation}
\norm{f}_{\BMO}=\max\{\norm{P_+f}_{\BMO},\norm{P_-f}_{\BMO}\}.
\label{z6}
\end{equation}
Finally, we recall that the subclass $\VMO\subset\BMO$ is the closure of all continuous functions in the $\BMO$-norm.

Let us comment on our choice of the norm in BMO.
Definition \eqref{eq:BMO} is absolutely crucial for our approach, as it ensures the
connection of rational approximations with Hankel operators via the AAK theorem. 
On the other hand, the details of the definition \eqref{eq:BMBM} 
are less important: the term $\abs{\wh{f}(0)}$ is inessential and the other two quantities 
in the right-hand side can be  combined in various ways. 
Our choice \eqref{eq:BMBM}  is motivated by the fact that it simplifies the expressions for
some coefficients appearing in the asymptotic formulas below.  
We could have chosen, for example, 
the following alternative definition of the BMO norm: 
$$
\norm{f}_{\BMO}^*  
=  
\norm{P_{-} f}_{\BMO}+\norm{P_{-}\bar{f}}_{\BMO}+\abs{\wh{f}(0)} ;
$$
this would only change the constant in the right-hand side of some asymptotic formulas, such as \eqref{z5} below.

\subsection{Rational approximation}
We denote by $\calR_n$ the set of all rational functions 
of degree $\leq n$ in the complex plane without poles on $\bbT$ and set
$$
\calR_n^{\pm}=\calR_n\cap H^2_\pm=P_\pm(\calR_n).
$$
Notice that $\calR_0^{+}=\{\const\}$, while $\calR^{-}_0=\{0\}$.

For $\omega\in\BMO$ and $n\geq0$, we define
\begin{equation}
\begin{split}
\rho_n(\omega)&=\dist_{\BMO}\{\omega,\calR_n\}:= \min\{ \|\omega -r  \|_{\BMO}: r\in \calR_n\},
\\
\rho_n^+(\omega)&= \dist_{\BMO}\{P_+\omega,\calR_n\} = \dist_{\BMO}\{P_+\omega,\calR^{+}_n\},
\\
\rho_n^-(\omega)&= \dist_{\BMO}\{P_-\omega,\calR_n\} = \dist_{\BMO}\{P_-\omega,\calR^{-}_n\}.
\end{split}
\label{eq:DR}
\end{equation}
There are some simple identities relating the quantities $\rho_n$, $\rho_n^+$, $\rho_n^-$,
see  Lemma~\ref{AAKXm} below. 
It is clear that $\omega\in\VMO$ (resp. $P_\pm\omega\in\VMO$)
if and only if $\rho_n(\omega)\to0$ (resp. $\rho_n^\pm(\omega)\to0$) as $n\to\infty$.

From our choice of the $\BMO$-norm it follows that $\rho_n^-(\omega)$ can be alternatively written as
\begin{equation}
\rho_n^-(\omega)=\dist_{L^\infty}\{\omega, \calR^{-}_n +H^2_+\}.
\label{eq:N-T}
\end{equation}
The problem of approximation by functions of the class
$\calR _n ^-  +H^2_+$ in $L^\infty$-norm is known as the Nehari-Takagi problem. 
Of course, a similar statement is true for $\rho_n^+(\omega)$ (see Lemma~\ref{lma.z1} below):
\begin{equation}
\rho_n^+(\omega)=\dist_{L^\infty}\{\omega,\calR^+_n +H^2_-\}.
\label{eq:N-T+}
\end{equation}

\subsection{Outline of results}
To give the flavour of our main result, first
we consider the following model functions:
\begin{align}
\omega_0 (e^{i\theta})&= \abs{\log\abs{\theta}}^{1-\alpha}\chi_0(\theta), & \theta\in[-\pi,\pi),
\label{z10}
\\
\omega_\pm (e^{i\theta})&=  \abs{\log\abs{\theta}}^{-\alpha}\chi_0(\theta) \1_\pm(\theta), & \theta\in[-\pi,\pi).
\label{z11}
\end{align}
Here $\alpha>0$ is a fixed parameter, 
$\1_\pm$ is the characteristic function of the semi-axis $\bbR_\pm$, 
and $\chi_0$ is a smooth even cutoff function, whose role is to remove the ``undesired'' 
singularity of the functions \eqref{z10} and \eqref{z11} at $\abs{\theta}=1$. 
More precisely, $\chi_0\in C_0^\infty(\bbR)$ is a function
which vanishes identically for $\abs{\theta}>c$ with some $c<1$ 
and such that $\chi_0(\theta)=1$ 
in some neighborhood of zero.

The following statement is a particular case of Theorem~\ref{nonAnalytic} below. 
We set
\begin{equation}
\varkappa(\alpha)=2^{-\alpha}
\pi^{1-2\alpha}
B(\tfrac1{2\alpha},\tfrac12)^\alpha
\label{a7}
\end{equation}
where $B(\cdot, \cdot)$ is the Beta function.

\begin{theorem}\label{thm.b1}
Let $\alpha>0$ and let 
$$
\omega(\mu)= v_{0}\omega_{0}(\mu)+ v_{+}\omega_{+}(\mu)+v_{-}\omega_{-}(\mu),
$$
where $v_{0}, v_{+} , v_{-} $  are arbitrary complex numbers. Put
$$
{  b}^{\pm}
=
\tfrac12(1-\alpha)v_0
\pm
\tfrac{1}{2\pi i } (v_{+}    - v_{-}  ). 
$$ 
Then
\begin{align}
\lim_{n\to\infty} n^\alpha \rho_n^\pm( \omega )
&=
\varkappa(\alpha) |{  b}^{\pm}|,
\notag
\\
\lim_{n\to\infty} n^\alpha \rho_n( \omega)
&=
\varkappa(\alpha)
\bigl(|{ b}^{+}|^{1/\alpha}+ |{ b}^{-}|^{1/\alpha}\bigr)^{\alpha}.
\label{z5}
\end{align}
\end{theorem}

\begin{remark}
\begin{enumerate}[1.]

\item
Similarly to Stahl's formula \eqref{eq:St}, the asymptotic coefficient $\varkappa(\alpha)$ in Theorem~\ref{thm.b1} is quite explicit. 
In contrast to this, it is not known whether the Bernstein constant ${\sf b}(\alpha)$ 
in \eqref{eq:Bern} can be expressed in terms of standard transcendentals; 
see,  e.g. \cite{Lu} for more on this issue. 
   
\item
For $0<\alpha<1$, the function $\omega_0$ is not in $L^\infty$,  although $\omega_0\in\VMO$.

\item
For $\alpha=1$, the function $\omega_0=\chi_0\in C^\infty$.
This agrees with the fact that in this case $\omega_0$  makes no contribution to ${b}^{\pm}$. 

\item
For $\alpha=0$, the functions $\omega_0$ and $\omega_\pm$
are in BMO but not in VMO.
Thus, in this case $\rho_n(\omega)$ (and $\rho_n^\pm(\omega)$) do not tend to zero as $n\to\infty$. 
This shows that one cannot go beyond $\alpha>0$ in Theorem~\ref{thm.b1}.

\item
Consider the case $v_0=0$; then the asymptotic coefficients ${ b}^{\pm}$ vanish precisely when 
$v_+=v_-$, i.e. when $\omega(e^{i\theta})$ is an even function of $\theta$.
This partially explains the fact that although the even function 
$\omega_0$ is more singular than $\omega_+$ and $\omega_-$, 
the rational approximations of these functions have the same power rate of convergence.

\item
Comparing Theorem~\ref{thm.b1} with Peller's result \eqref{eq:BesPel}, 
we see that the functions $\omega_0$ and $\omega_\pm$ 
belong to the Besov-Lorentz  class $\mathfrak{B}^{\alpha}_{1/\alpha,\infty}$, but do not belong 
to the class $\mathfrak{B}^{\beta}_{1/ \beta,\infty}$ for any $\beta>\alpha$. This yields
explicit (and sharp!)  examples of functions in these Besov-Lorentz classes.

\item
The analogue 
of Theorem~\ref{thm.b1} for $L^\infty$-distances remains an open problem.
\end{enumerate}
\end{remark}

For functions $\omega\in \BMO \cap \, H^2_\pm$  , the distances $\rho_{n}^\pm (\omega) = \rho_{n}  (\omega)$ do not depend on the choice of the norm  \eqref{eq:BMBM} --- see relations \eqref{eq:N-T} and \eqref{eq:N-T+}. Therefore for functions analytic for $|z| <1$ (or for $|z| >1$),  Theorem~\ref{thm.b1} is stated in a quite intrinsic form.
Consider, for example, the function
\begin{equation}
\omega (z)= (-\log (1-z) +c)^{1-\alpha}, \quad \alpha > 0,
\label{eq:RR1}
\end{equation}
where a number $c$ is chosen in such a way that $\log (1-z) \neq c$ for all $z$ with $\abs{z}\leq 1$. 
Then $\omega(z)$ is analytic in the unit disk $\bbD$ and is singular only at the point $z=1$ on the unit circle. 

The following statement is a particular case of Theorem~\ref{Analytic} below. 
\begin{theorem}\label{RR}
Let   the function $\omega(z)$   be defined  by formula \eqref{eq:RR1}. Then there exists
$$
\lim_{n\to\infty} n^\alpha  \rho_{n}^+ (\omega) = |1-\alpha| \varkappa(\alpha).
$$
 \end{theorem}
 
Two sided estimates by $c\, n^{-\alpha}$ of  $ \rho_{n}^+ (\omega)$ for the function \eqref{eq:RR1} (and for more general functions of this type) are known. They were obtained by A.~A.~Pekarski\u{\i} in \cite{Pekarski1} (see Example~2.2). 
Later (see \cite{Pekarski},  relation (31)) Pekarski\u{\i} also proved the upper bound in the $L^\infty$-norm, 
$\dist_{L^\infty}\{\omega,  \calR^{+}_n\} =O (n^{-\alpha})$.

We emphasize that our main results, Theorems~\ref{HT} and \ref{Analytic} below, allow for an arbitrary finite number of logarithmic singularities of $\omega  $ on the unit circle.

\subsection{Some ideas of the approach}

We start by recalling some concepts related to Hankel operators; 
for the details, see, e.g., the books \cite{NK, Peller}.
For $\omega\in L^2$, the Hankel operator 
$K (\omega): H^2_+ \to H^2_-$ is defined by the formula
\begin{equation}
K(\omega)f=P_-(\omega f) .
\label{eq:HH}
\end{equation}
In this context, $\omega$ is called the \emph{symbol} of $K(\omega)$. 
The definition \eqref{eq:HH} makes sense, for example, on  all  
polynomials $f$. It is evident that $K(\omega)$ depends only on the  
  part $P_{-}\omega$ of $\omega$, i.e. 
$
K(\omega)=K(P_-\omega)$.
Nehari's theorem 
ensures that $K(\omega)$ is a bounded
operator if and only if $P_{-}\omega\in\BMO$, and Hartman's theorem 
says that $K(\omega)$ is compact if and only if $P_-\omega\in \VMO$;  see Proposition~\ref{Nehari} below.

Another equivalent point of view on Hankel operators appears when one considers
the matrix representation of $K(\omega)$ with respect to the standard bases in $H^2_\pm$. 
Consider the bases $\{\mu^j\}_{j=0}^\infty$ in $H^2_+$ and $\{\mu^{-1-k}\}_{k=0}^\infty$ in $H^2_-$. 
Then the matrix representation of $K(\omega)$ with respect to this pair of bases is 
$$
(K(\omega)\mu^j,\mu^{-1-k})_{L^2}
=
\wh \omega(-1-j-k), \quad j,k\geq0.
$$
It will be convenient to have a separate piece of notation for such infinite matrices, considered as
 operators on the sequence space $\ell^2:=\ell^2(\bbZ_+)$.
Given a sequence $\{h(j)\}_{j=0}^\infty$ of complex numbers, we define the 
Hankel operator $\Gamma(h)$ on  $\ell^2 $ by 
\begin{equation}
(\Gamma(h)u)(j)=\sum_{k=0}^\infty h(j+k)u(k).
\label{eq:a5}
\end{equation}
Now suppose $\omega\in L^2$ and $P_-\omega\in\VMO$; take $h(j)=\wh\omega(-1-j)$ for all $j\geq0$. 
Then the operators $K(\omega)$ and $\Gamma(h)$ have the same matrix representation with respect
to some pairs of orthonormal  bases, and hence $\Gamma(h)=\mathcal{ U}_{-}^*K(\omega)\mathcal{ U}_+$ for appropriate  unitary mappings $\mathcal{ U}_\pm : {\ell}^2 \to H^2_{\pm}$.
It follows that these operators have the same sequence of 
singular values (see Section~\ref{sec.back1}):
$$
s_n(K(\omega))=s_n(\Gamma(h)), \quad \forall n\geq0,
\quad
\text{ if } \quad h(j)=\wh\omega(-1-j), \quad \forall j\geq0 .
$$

The proof of our main result relies on the following two ingredients:
\begin{itemize}

\item
The Adamyan-Arov-Krein (AAK) theorem. One of the alternative ways to state this theorem is to say that  
$$
s_{n}(K(\omega))=\rho_n^- ( \omega), \quad n\geq0,
$$
if  $P_{-}\omega\in \VMO$.
We give some background related to this formula in Section~\ref{sec.back}.

\item
Our results of \cite{II, LOC}, which give an asymptotic formula for the singular values
of a class of Hankel operators $\Gamma(h)$. 
Those are the operators corresponding to the sequences $h$ of the form
\begin{equation}
h(j)= j^{-1}(\log j)^{-\alpha} + \text{error term}, \quad j\to\infty,
\label{z90}
\end{equation}
and, more generally, 
\begin{equation}
h(j)=\sum_{\ell=1}^L b_\ell j^{-1}(\log j)^{-\alpha}\zeta_\ell^{-j} + \text{error term}, \quad j\to\infty,
\label{z9}
\end{equation}
where $b_1,\dots,b_L\in\bbC$ and 
$\zeta_1,\dots,\zeta_L\in\bbT$.
Sequences of the type \eqref{z90} are required in the proof of Theorems~\ref{thm.b1} and \ref{RR}, 
while sequences of the type \eqref{z9} are required in the proof of the more general Theorems~\ref{nonAnalytic}
and \ref{Analytic}, which pertain to the functions $\omega$ with several ($=L$) singularities on the unit circle.
\end{itemize}

Our construction depends on the interplay between   two representations of Hankel
operators: as $K(\omega):H^2_+\to H^2_-$ and as $\Gamma(h):\ell^2\to\ell^2$. 
From the technical point of view, we only have to relate the class of Hankel operators $K(\omega)$, 
where the symbol $\omega$ has finitely many ($=L$) logarithmic singularities on $\bbT$, to the 
class of Hankel operators $\Gamma(h)$, where $h$ is of the form \eqref{z9}. This requires a rather careful analysis
of the Fourier coefficients of such functions $\omega$. We show that
every singularity of $\omega$ generates one of the  terms in the right-hand side of \eqref{z9}. 
This result is stated as Theorem~\ref{lma.b1}.

\subsection{The structure of the paper}
In Section~\ref{sec.back} we recall some background information related to the theory of Hankel operators
and to the AAK theorem. 
In Section~3 we state our main results, Theorems~\ref{nonAnalytic}
and \ref{Analytic}, which are extensions of Theorems~\ref{thm.b1} and \ref{RR}.
In the same Section, 
we deduce our main results from the technical Theorem~\ref{lma.b1}, which describes the 
asymptotic behaviour of the Fourier coefficients of functions with   logarithmic singularities on $\bbT$. 
The proof of Theorem~\ref{lma.b1} is given in 
Section~4.

\section{Background information}\label{sec.back}

\subsection{Schatten classes}\label{sec.back1}
Here we briefly recall some background facts on Schatten classes; for a
detailed presentation, see,  e.g.  the book \cite{BSbook}.
 Let $\calB$ be the algebra of bounded operators on a Hilbert space $\mathcal H$, and let   $\norm{\cdot}$ be the operator norm.   
 Singular values of a compact operator $A\in\calB$ are defined by the relation 
$s_{n} (A)=\lambda_{n} (\abs{A})$, 
where $\{\lambda_{n} (\abs{A})\}_{n=0}^\infty$
is the non-increasing sequence of  eigenvalues of the compact 
non-negative operator $\abs{A}=\sqrt{A^* A}$ (with multiplicities taken into account).
Singular values may  also be defined by the relation
\begin{equation}
s_{n} (A)= \min \{\| A- B\| : B\in \calB, \quad \rank B\leq n\}, \quad n=0,1, \ldots. 
\label{eq:rank}
\end{equation}

For $p>0$, the Schatten class $\Sch_p$ and the weak  Schatten class $\Sch_{p,\infty}$ of compact operators  
are defined by the conditions 
\begin{align*}
A\in\Sch_p 
\quad &\Leftrightarrow \quad 
\sum_{n=0}^\infty s_n(A)^p<\infty,
\\
A\in\Sch_{p,\infty} 
\quad &\Leftrightarrow \quad 
\sup_{n\geq0} (n+1)^{1/p}s_n(A)
<\infty.
\end{align*}
Of course, we have $\Sch_p \subset \Sch_{p,\infty}$.

\subsection{Relations between $\rho_n(\omega)$, $\rho_n^+(\omega)$, $\rho_n^-(\omega)$}

Recall that $\calR_n$ consists of all rational functions with at most $n$ poles, including the pole at infinity,
but with no poles on the unit circle $\bbT$; the poles are counted with 
multiplicities taken into account.
Thus, $r\in\calR_n$ if and only if
$$
r(z)= p (z)+ \sum_{|z_{j}|\neq 1}\sum_{\ell=1}^{L_{j}} c_{j,\ell} (z-z_{j})^{-\ell},
$$
where $p$ is a polynomial and 
$$
\deg r =\deg p+ L_{1}+ L_{2}+\cdots \leq n. 
$$
Hence the functions $r_{+}= P_{+} r\in \calR^{+}_n$ and $r_-=P_- r\in \calR^{-}_n$ are given by 
\begin{equation}
r_+(z)=p (z)+ \sum_{|z_{j}|>1}\sum_{\ell=1}^{L_{j}} c_{j,\ell} (z-z_{j})^{-\ell},
\quad
r_-(z)=\sum_{|z_{j}|<1}\sum_{\ell=1}^{L_{j}} c_{j,\ell} (z-z_{j})^{-\ell}.
\label{eq:ratii}
\end{equation}
The following simple relations between the distances 
$\rho_n(\omega)$, $\rho_n^+(\omega)$ and $\rho_n^-(\omega)$ (see \eqref{eq:DR}) will be useful.

\begin{lemma}\label{lma.z1}
For any $\omega\in\VMO$ and any $n\geq0$, we have the relation
\begin{equation}
\rho_n^+(\overline{\omega}) = \rho_n^-(\omega).
\label{z1}
\end{equation}
Moreover, formula \eqref{eq:N-T+} holds true.
\end{lemma}

\begin{proof} 
Put $H^2_{+} (0)= \{ f\in  H^2_{+} : \wh f (0)=0\}$ and $\calR^{+}_n (0)= \{ r\in  \calR^{+}_n : \wh r (0)=0\}$. Then
$
\ov{H^2_{+} (0)}= H^2_-, \quad \ov{\calR^{+}_n(0)}= \calR^{-}_n
$
and $H^2_{+} =H^2_{+} (0)+ {\bbC}$, $\calR^{+}_n =\calR^{+}_n  (0)+ {\bbC}$.
Therefore, by  the definition \eqref{eq:DR} of $\rho_n^+$,  we have
$$
\rho_n^+(\overline{\omega})
=
\min\{\norm{P_+ (\overline{\omega}-r_{+}-r_{0})}_{\BMO}: r_{+}\in\calR^{+}_n(0), r_{0}\in {\bbC}\}
$$
 which in view of the relation \eqref{z7} yields
$$
\rho_n^+(\overline{\omega})
=
\max\bigl\{\min\{\norm{P_- \omega- r_-   }_{\BMO}: 
r_-\in\calR^{-}_n\}, \min\{    \abs{\wh\omega(0)-r_0} : 
 r_0\in\bbC\}\bigr\} .
$$
Choosing $r_0=\wh\omega(0)$, we see that the right-hand side here equals $\rho_n^-(\omega)$.

Putting together relations \eqref{eq:N-T} and \eqref{z1} and passing to the complex conjugation, we see that
$$
\rho_n^+ (\omega)=\dist_{L^\infty}\{\overline{\omega} , \calR^{-}_n +H^2_+\}=\dist_{L^\infty}\{ \omega , \overline{\calR^{-}_n} +\overline{H^2_+}\}.
$$
Since
$$
 \overline{\calR^{-}_n} +\overline{H^2_+}=  \calR^{+}_n (0) +\overline{H^2_{+} (0)+ {\bbC}}=  \calR^{+}_n (0) + H^2_- + {\bbC},
$$
we obtain formula \eqref{eq:N-T+}.
\end{proof} 

\begin{lemma}
For any $\omega\in\VMO$ and any $n\geq0$, we have the relation
\begin{equation}
\rho_n(\omega)
=
\min\bigl\{\max\{\rho_{n_+}^+(\omega),\rho_{n_-}^-(\omega)\}: n_++n_-=n\bigr\}.
\label{z1a}
\end{equation}
\end{lemma}

\begin{proof} 
For $r\in \calR_n$, denote $r_\pm=P_\pm r$. 
From \eqref{eq:ratii},
it is easy to see that $r_+\in\calR^{+}_{n_+}$ and $r_-\in\calR^{-}_{n_-}$, where
$n_++n_-=n$. 
Conversely, if $r_\pm\in\calR^{\pm}_{n_\pm}$, then $r=r_++r_-\in\calR_n$
with $n=n_++n_-$. 
By the identity \eqref{z6}, we have
$$
\norm{\omega-r}_{\BMO}
=
\max\{\norm{P_+\omega-r_+}_{\BMO},\norm{P_-\omega-r_-}_{\BMO}\}.
$$
It follows that 
$$
\rho_n(\omega)
=
\min\bigl\{\max\{\norm{P_+\omega-r_+}_{\BMO},\norm{P_-\omega-r_-}_{\BMO}\}:
r_\pm\in\calR_{n_\pm}^{\pm}, n_++n_-=n\bigr\};
$$
the right-hand side here coincides with the right-hand side in   \eqref{z1a}.
\end{proof}

It will be convenient to rewrite  \eqref{z1a} in terms of the following counting functions:  
\begin{equation}
\nu(\omega;s)=\#\{n\geq0: \rho_n(\omega)>s\}, 
\quad
\nu^\pm(\omega;s)=\#\{n\geq0: \rho_n^\pm(\omega)>s\}, 
\label{nu}
\end{equation}
where $s>0$.

\begin{lemma}\label{CF}
For any $\omega\in\VMO$ and any $s >0$, the relation 
\begin{equation}
\nu(\omega;s)=\nu^+(\omega;s)+\nu^-(\omega;s)
\label{z1b}
\end{equation}
holds true.  
\end{lemma}

\begin{proof}
Fix $s>0$. 
Observe that $\rho_n(\omega)\leq s$ is equivalent to 
$\nu(\omega;s)\leq n$, and similarly for $\rho_n^\pm(\omega)$. 
By \eqref{z1a}, for any $n\geq0$, the relation $\rho_n(\omega)\leq s$ is equivalent to 
$$
\exists n_+, n_-: \quad n_++n_-=n, \quad \rho_{n_+}^+(\omega)\leq s \text{ and }  \rho_{n_-}^-(\omega)\leq s.
$$
This can be rewritten as
$$
\exists n_+, n_-: \quad n_++n_-=n, \quad \nu^+(\omega;s)\leq n_+ \text{ and }  \nu^-(\omega;s)\leq n_-,
$$
which is equivalent to $\nu^+(\omega;s)+\nu^-(\omega;s)\leq n$. 
We have proven that $\nu(\omega;s)\leq n$ is equivalent to $\nu^+(\omega;s)+\nu^-(\omega;s)\leq n$; 
thus, we get \eqref{z1b}.  
\end{proof}
 
\subsection{Hankel operators on Hardy spaces}

Here we recall several fundamental results of the theory of Hankel operators. 
The first proposition below is Nehari's theorem \cite{Nehari}, which we
combine  for convenience with the compactness result due to P.~Hartman \cite{Hartman}. 
  
\begin{proposition}\cite[Theorems~1.1.3 and 1.5.8]{Peller}\label{Nehari}
Suppose that $\omega\in L^2 $. Then the Hankel operator $K (\omega): H^2_+ \to H^2_-$   
 is bounded \emph{(\emph{resp. compact})} if and only if
 $P_{-}\omega\in \BMO$ \emph{(\emph{resp.  $P_{-}\omega\in \VMO$})}. Moreover,
 \begin{equation}
\| K(\omega)\| = \| P_{-} \omega\|_{ \BMO }.
\label{eq:BMO1}
\end{equation} 
\end{proposition}

In view of  Proposition~\ref{Nehari}, the  definition \eqref{eq:BMBM} of the BMO norm can be rewritten as
$$
\norm{\omega}_{ \BMO } =\max\{ \norm{K(\omega)}, \norm{K(\overline{\omega})}, \abs{\wh{\omega} (0)}\}.
$$

The Kronecker theorem   describes all finite rank Hankel operators.  

\begin{proposition}
A Hankel operator $K(\omega)$ has rank $n$ if and only if $P_{-}\omega\in {\calR}_{n}$
\emph{(\emph{equivalently, if and only if $P_{-}\omega\in {\calR}_n^{-}$})}. 
\end{proposition}

 The Adamyan-Arov-Krein theorem states that, for Hankel operators, the minimum in \eqref{eq:rank} can be   taken   over Hankel operators only. We denote by $\mathcal{K}$ the set of all bounded Hankel operators.

\begin{proposition}\label{AAK}\cite[Theorem 0.1]{AAK}
Let  $K $ be a compact Hankel operator. Then 
$$
s_{n} (K)= \min \{\| K-G\| : G\in \mathcal{K}, \rank G\leq n\}, \quad n=0,1, \ldots. 
$$
\end{proposition}

Combining  Kronecker   and Adamyan-Arov-Krein theorems and taking into account relation \eqref{eq:BMO1}
and Lemma~\ref{lma.z1}, one obtains the following result
(which is essentially Theorem 0.2 in \cite{AAK}).

\begin{proposition}\label{AAKXm}
Let  $\omega\in \VMO$. Then for all $n\geq0$,
$$
\rho_n^+(\omega)=s_{n}(K(\overline{\omega})),
\quad
\rho_n^-(\omega)=s_{n}(K(\omega)).
$$
\end{proposition} 

Thus, the problem of rational approximation of a function $\omega\in \VMO$ is equivalent 
to the study of the singular values of the corresponding Hankel operator.

\subsection{Schatten class properties of Hankel operators}
Here we recall   important results due to V.~Peller that characterise Hankel operators of 
Schatten classes. Some partial results in this direction were independently obtained by S.~Semmes in 
\cite{Semmes} and by A.~A.~Pekarski\u{\i} in
\cite{Pekarski1}.
Definitions of the Besov class $B^{1/p}_{p,p}$ and the Besov-Lorentz classe $\mathfrak B^{1/p}_{p,\infty}$
are given in the Appendix; further relevant information can be found in Peller's book 
\cite{Peller}. 
We will not need the three propositions below in our construction, and they are given here
only in order to put our results into the right context. 

\begin{proposition}\cite[Corollaries~6.1.2, 6.2.2 and 6.3.2]{Peller}
 Let  $\omega\in L^2  $ and $p>0$. Then  the Hankel operator $K (\omega)$ belongs to the Schatten class $\Sch_p$   if and only if
 $P_{-}\omega\in B^{1/ p}_{p,p} $.
\end{proposition}

In view of Proposition~\ref{AAK}, this  implies the following result on rational approximation in the $\BMO$ norm.

\begin{proposition}\cite[Theorem~6.6.1]{Peller}
Let $\omega\in \VMO$ and $p>0$. 
Then  the condition
$$
\bigl\{\dist_{\BMO}\{\omega ,\calR_{n}\}\bigr\}_{n=0}^\infty
\in \ell^p({\bbZ}_{+})
$$
is satisfied if and only if
 $ \omega\in B^ {1/p}_{p,p}$.
\end{proposition}

Using real interpolation, Peller has also obtained the   ``weak version'' of this result
(see \cite[Section 6.4]{Peller}). 

\begin{proposition}
 Let  $\omega\in \VMO$ and $p>0$. 
Then the condition 
$$
\dist_{\BMO}\{\omega ,\calR_{n}\}=O(n^{-1/p}), \quad n\to\infty,
$$
is satisfied if and only if
$\omega\in \mathfrak{B}^{1/ p}_{p,\infty}$.
\end{proposition}

\subsection{Hankel operators in $\ell^2$}

Here we state the result of \cite{LOC} on the asymptotics of singular values of Hankel operators $\Gamma (h)$.
First we need some notation. 
For a sequence $g=\{g(j)\}_{j=0}^\infty$, we define iteratively the sequences
$g^{(m)}=\{g^{(m)}(j)\}_{j=0}^\infty$, $m=0,1,2,\dots$, by setting $g^{(0)}(j)=g(j)$  
and 
$$
g^{(m+1)}(j)=g^{(m)}(j+1)-g^{(m)}(j), 
\quad j\geq0.
$$

\begin{theorem}\cite[Theorem~3.1]{LOC}\label{thm.a4}
Let $\alpha>0$, let $\zeta_1, \zeta_2,\dots,\zeta_L\in\bbT$ be distinct numbers, and let 
${b}_{1}, {  b}_2,\dots, {  b}_L\in\bbC$.
Let $h$ be a sequence of complex numbers such that
\begin{equation}
h(j)=\sum_{\ell=1}^L \bigl( {  b}_\ell j^{-1}(\log j)^{-\alpha}+ g_\ell(j)\bigr)\zeta_\ell^{j+1},
\quad
j\geq2,
\label{a13}
\end{equation}
where the error terms $g_\ell$, $\ell=1,\ldots, L$, obey the estimates
\begin{equation}
g_\ell^{(m)}(j)
=o(j^{-1-m}(\log j)^{-\alpha}), 
\quad j\to\infty, 
\label{a5a}
\end{equation}
for all   $m=0,1,\dots $.
Then  the Hankel
operator $\Gamma (h)$ defined by formula \eqref{eq:a5}   is compact in $\ell^2$, and its singular values  
  satisfy the asymptotic relation
\begin{equation}
s_n(\Gamma (h))= {  a} \, n^{-\alpha}+o(n^{-\alpha}), 
\quad
{ a} =\varkappa(\alpha)\Bigl(\sum_{\ell=1}^L \abs{{  b}_\ell}^{1/\alpha}\Bigr)^\alpha,
\label{z0c}
\end{equation}
as $n\to\infty$, 
where
the coefficient 
$\varkappa(\alpha)$ is given by formula \eqref{a7}.
\end{theorem}

\begin{remark*}
In fact, it suffices to require condition \eqref{a5a} for $0\leq m\leq M(\alpha)$, where 
$M(\alpha)$ is an explicit finite number. 
\end{remark*}

\section{Main results}

The structure of this section is as  follows. 
First, we state a technical Theorem~\ref{lma.b1}, which gives the asymptotics of the Fourier coefficients
for functions $\omega$ with  logarithmic singularities on the unit circle. 
The proof of this theorem will be provided in the next section. 
Then, using this theorem, we prove Theorem~\ref{HT}, which  
yields the asymptotics of the singular values for Hankel operators $K(\omega)$ with
$\omega$ as above. 
Finally, we state and prove our main results (Theorems~\ref{nonAnalytic} and \ref{Analytic}) on rational approximation
of such functions $\omega$. 
They are obtained as simple corollaries of Theorem~\ref{HT}. 

\subsection{Fourier coefficients of singular functions}

Here we consider functions $\omega(\mu)$ which are smooth on the unit circle except at the point $\mu=1$, where 
$\omega(\mu)$ have logarithmic singularities. 
These singularities will be slightly more general than those of the ``model functions'' $\omega_0$, $\omega_\pm$
of Section~\ref{sec.z}
(see \eqref{z10}, \eqref{z11}) and will contain additional functional parameters.

As in Section~\ref{sec.z}, we fix an even function 
$\chi_0\in C^\infty(\bbR)$  satisfying the condition
\begin{equation}
\chi_0(\theta)=
\begin{cases}
1& \text{for $\abs{\theta}\leq c_1$,}
\\
0& \text{for  $\abs{\theta}\geq c$,}
\end{cases}
\label{d7}
\end{equation}
where $0<c_1<c $ are sufficiently small numbers (we will be more specific below). 
First let us informally discuss the structure of an admissible singularity of $\omega$ 
at the point $\mu=1$ of the unit circle. 
Below the index $\sigma$ takes values $+$ and $-$ and $\1_\sigma$ denotes the characteristic
function of the semi-axis $\bbR_\sigma$. The more general version of $\omega_0$ (see \eqref{z10}) 
is the function
$$
\sum_{\sigma=\pm} v_{0,\sigma}(\theta)(-\log\abs{\theta}+u_{0,\sigma}(\theta))^{1-\alpha}\1_\sigma(\theta)\chi_0(\theta) 
$$ 
where $v_{0,\sigma}$ and $u_{0,\sigma}$ are arbitrary complex valued $C^\infty$   functions such that
\begin{equation}
 v_{0,+}  (0) =v_{0,-}  (0)=:v_{0} .
 \label{eq:bbx}
\end{equation}
Similarly, the generalisation of the linear combination of  $\omega_+$ and $\omega_-$ (see \eqref{z11}) is 
$$
\sum_{\sigma=\pm} v_{1,\sigma}(\theta)(-\log\abs{\theta}+u_{1,\sigma}(\theta))^{-\alpha}\1_\sigma(\theta)\chi_0(\theta) 
$$ 
with some $C^\infty$   functions $v_{1,\sigma}$ and $u_{1,\sigma}$. 
Below we combine these two expressions more succinctly as a sum of four terms. 
More precisely, we introduce the following assumption.

\begin{assumption}\label{AsLog}
Let $\alpha>0$, and let $v_{j, \sigma}(\theta)$ and $u_{j, \sigma}(\theta)$, $j=0,1$, $\sigma=\pm$, be complex valued $C^\infty$ functions of $\theta\in \bbR$ such that
condition \eqref{eq:bbx} is satisfied.
Then the function $\omega$ is defined by  the relation
\begin{equation}
\omega(e^{i\theta})
=  
\sum_{j=0,1}\sum_{\sigma=\pm} 
v_{j, \sigma} (\theta)(-\log\abs{\theta} +  u_{j, \sigma }(\theta))^{1 -j-\alpha}\1_\sigma(\theta) \chi_0(\theta), 
\quad \theta\in(-\pi,\pi].
\label{b4cx}
\end{equation}
\end{assumption}

Here $c_{2}$ is chosen so small   that $\theta=0$ is the only singularity of  the functions in the sum \eqref{b4cx},
that is, 
$$
-\log\abs{\theta} +  u_{j, \sigma }(\theta)\neq 0  \quad \text{ if } \quad \theta\in [-c_{2}, c_{2}]  
$$
for $ j=0,1$, $\sigma=\pm$. 
The function $z^{j-\alpha}$ for $z=-\log\abs{\theta} +  u_{j,\sigma} (\theta)$ 
in \eqref{b4cx} is defined by the principal branch, $z^{j-\alpha}=e^{(j-\alpha)\log z}$, where 
we assume that 
$$
\arg (-\log\abs{\theta} +  u_{j,\sigma} (\theta))\to 0 \quad \text{ as }  \quad \theta\to 0
$$
for all these functions.

For a function $\omega$ satisfying Assumption~\ref{AsLog}, we put
\begin{equation}
b=b(\omega)=
 (1-\alpha)v_{0} \bigl(  \tfrac12 - \tfrac {1}{2 \pi i }  (u_{0,+ } (0)- u_{0,- } (0)) \bigr)
 -
 \tfrac {1}{2 \pi i }  (v_{1,+}  (0) - v_{1,-} (0) ) .
\label{b4ab}
\end{equation}

The analytic core of our construction is the following theorem. 
\begin{theorem}\label{lma.b1}
Under Assumption~$\ref{AsLog}$,  the Fourier coefficients   of $\omega (\mu)$ admit  the   representation
\begin{equation}
\wh\omega (- j)
=
b  j^{-1}(\log j)^{-\alpha}+g(- j), \quad j\geq 2,
\label{b5}
\end{equation}
where the coefficient $b=b(\omega)$ is given by formula \eqref{b4ab} and the error term $g (-j)$ satisfies the estimates 
\begin{equation}
 g ^{(m)} (- j)=O \bigl(j^{-1-m}(\log j)^{-\alpha-1}\bigr), \quad j \to+\infty,
\label{eq:c8F}
\end{equation}
for all $m\geq0$. 
\end{theorem}
  
We emphasize that the leading terms of the asymptotics of the Fourier coefficients of these functions
depend on the combination \eqref{b4ab} only. 

The proof of Theorem~\ref{lma.b1} will be given in the next section.

\subsection{Hankel operators with singular symbols}

Here we state a result about the singular value asymptotics for Hankel operators $K(\omega)$ with
symbols  $\omega$ having finitely many 
logarithmic singularities.

\begin{theorem}\label{HT} 
Let   $\zeta_1, \zeta_2,\dots,\zeta_L\in\bbT$ be distinct numbers, and  let the functions $\omega_{1}, \omega_2,\dots, \omega_L$ satisfy Assumption~$\ref{AsLog}$. Define the function
\begin{equation}
\omega(\mu)
= 
\sum_{\ell=1}^L \omega_{\ell}(\mu/ \zeta_{\ell}), \quad \mu\in\bbT,
\label{omega}
\end{equation}
and set
\begin{equation}
a(\omega)   =\varkappa(\alpha)\Bigl(\sum_{\ell=1}^L |b(\omega_{\ell})|^{1/\alpha}\Bigr)^\alpha,
\label{eq:sfa}
\end{equation}
where the numbers $b(\omega_\ell)$ are given by \eqref{b4ab} and $\varkappa(\alpha)$  is  the coefficient  \eqref{a7}.  
Then  the Hankel operator
$K (\omega )$ is compact and its    singular values   have the asymptotics
$$
s_n(K (\omega))= a(\omega)  \, n^{-\alpha}+o(n^{-\alpha}), 
$$
as $n\to\infty$.   
\end{theorem}

\begin{proof}
Observe that for   arbitrary $\zeta\in{\bbT}$ and   $\phi \in L^1 ({\bbT})$, we have
$$
\wh{\phi_{\zeta}}(j)
= 
\wh{\phi} (j) \zeta^{-j}
\quad \text{ if } \quad 
\phi_{\zeta}(\mu)
=  
\phi (\mu/ \zeta).
$$
Therefore it follows from \eqref{omega} that 
$$
\wh\omega(-j-1)=\sum_{\ell=1}^L\wh \omega_\ell(-j-1)\zeta_\ell^{j+1}.
$$
Let $h(j)=\wh{\omega} (-j-1)$.
According to  Theorem~\ref{lma.b1}  the sequence  $h(j)$   satisfies condition \eqref{a13} as $j\to +\infty$; the corresponding asymptotic coefficients $b_{\ell}=b(\omega_\ell)$ are defined by formula \eqref{b4ab}. 
Thus Theorem~\ref{thm.a4} implies  the  asymptotic formula \eqref{z0c}    
for the singular values of the Hankel operator $\Gamma(h)$. 
Since $\Gamma(h)$ and $K(\omega)$ have the same set of singular values, we obtain the desired result. 
\end{proof}
 It is important that the singularities of the symbol \eqref{omega} are
located at distinct points $ \zeta_1,\dots, \zeta_L$.

\begin{remark}
Let $\omega$ be a function satisfying Assumption~\ref{AsLog}, but without the condition \eqref{eq:bbx}, 
i.e. with  $v_{0,+}  (0)\not=v_{0,-}(0)$;  simple examples of such function are
$$
\omega_{\pm}(e^{i\theta})
=
\abs{\log\abs{\theta}}^{1-\alpha}\1_{\pm}(\theta)\chi_0(\theta)
\quad \text{ or } \quad
\omega(e^{i\theta})
=\abs{\log\abs{\theta}}^{1-\alpha}\sign\theta\chi_0(\theta).
$$
Now the terms with $j=1$ in \eqref{b4cx} are inessential and instead of \eqref{b4ab}  we  put
$$
\wt b(\omega)=
 -
 \tfrac {1}{2 \pi i }  (v_{0,+}  (0) - v_{0,-} (0) ) .
$$

Let  $\omega$ be given by formula \eqref{omega} where each $\omega_{\ell}$ is as above. 
 Then the operator $K(\omega)$ is   compact for $\alpha>1$ only, and the asymptotics of its
singular values is of a different order:  
$$
s_n(K (\omega))= \wt a \, n^{-\gamma}+o(n^{-\gamma}), \quad \gamma=\alpha-1,
$$
where 
$$
\wt a  =\varkappa(\gamma)\Bigl(\sum_{\ell=1}^L |b(\wt \omega_{\ell})|^{1/\gamma}\Bigr)^\gamma.
$$

This fact follows from Theorem~\ref{HT} with $v_{0}=0$ and $\alpha$ replaced by $\alpha-1$.
\end{remark}

\subsection{Rational approximation}

We recall that the distances $\rho_n(\omega)$ and $\rho_n^\pm(\omega)$ are defined by relations \eqref{eq:DR}.
Our main result on rational approximation is 
\begin{theorem}\label{nonAnalytic}
Assume the hypothesis of Theorem~$\ref{HT}$ and set  
$$
a^{+}=a(\overline\omega),\quad a^{-}=a(\omega), \quad
{ a} =  \bigl( ({ a}^{+})^{1/\alpha}+({ a}^{-})^{1/\alpha}
\bigr)^\alpha.
$$
Then 
\begin{align}
\lim_{n\to\infty}n^\alpha
\rho_n^\pm(\omega)
&= {  a}^{\pm},
\label{LA1h}
\\
\lim_{n\to\infty}n^\alpha
\rho_n(\omega)
&= a.
\label{LA1H}
\end{align}
 \end{theorem}

\begin{proof}
To prove  \eqref{LA1h}, it suffices to put together Proposition~\ref{AAKXm} and Theorem~\ref{HT}.
In order to prove \eqref{LA1H}, we observe that \eqref{LA1h}
can be equivalently rewritten in terms of the counting functions $\nu^\pm(\omega;s)$ 
(see \eqref{nu}) as
$$
\lim_{s\to0} s^{1/\alpha}\nu^\pm(\omega;s)=({ a}^{\pm})^{1/\alpha}.
$$
It now  follows from Lemma~\ref{CF}  that
$$
\lim_{s\to0} s^{1/\alpha}\nu(\omega;s)=({  a}^{+})^{1/\alpha}+({ a}^{-})^{1/\alpha},
$$
which is equivalent to \eqref{LA1H}. 
\end{proof}

\begin{remark}\label{rmk.error}
  Theorem~\ref{HT}  automatically  extends to symbols $\omega$ that include an error term:
\begin{equation}
\omega(\mu)
= 
\sum_{\ell=1}^L \omega_{\ell}(\mu/ \zeta_{\ell})+\wt \omega(\mu), \quad \mu\in\bbT,
\label{eq:ass}\end{equation}
where $\wt\omega$ is any symbol such that
\begin{equation}
s_n(K(\wt\omega))=o(n^{-\alpha}), \quad n\to\infty.
\label{kyfan}
\end{equation}
This follows by a standard application of Ky Fan's lemma (see e.g. \cite[Section II.2.5]{GK}).
Condition \eqref{kyfan} is satisfied, for example, when $P_{-}\wt\omega\in B_{1/\alpha,1/\alpha}^\alpha$. 
Therefore  Theorem~\ref{nonAnalytic} is also true for functions \eqref{eq:ass} where $\wt\omega\in B_{1/\alpha,1/\alpha}^\alpha$. 
\end{remark}

Theorem~\ref{thm.b1} is a particular case of Theorem~\ref{nonAnalytic}, with the following choice of parameters:
$L=1$, $\zeta_1=1$, and 
$$
v_{0,\pm}(\theta)=  v_{0}, \quad 
v_{1,\pm}(\theta)=  v_{\pm},\quad 
u_{0,\pm}(\theta)=0,\quad
u_{1,\pm}(\theta)=  0.
$$
 
Formula \eqref{eq:sfa}
is quite intuitive from the viewpoint of singular value asymptotics. It means that the contributions of different singularities of  the symbol $\omega$ to the singular values counting function are independent of each other. On the other hand, this formula does not look obvious in the approximation theory framework.

\subsection{Rational  approximation of analytic functions}

Let us consider the case of $\omega(z)$ analytic in the unit disc; then $\rho_{n} (\omega)=\rho_{n}^+ (\omega)$. 
Let $u(z)$ be analytic in $\bbD$, $u\in C^\infty(\ov{\bbD})$; fix some $\zeta \in\bbT$ and
assume that 
\begin{equation}
-\log ( \zeta - z)+u (z)\neq 0, \quad z\in \ov{\bbD}.
\label{LA2a}
\end{equation}
Define
\begin{equation}
\omega (z)=\bigl(-\log ( \zeta - z)+u(z)\bigr)^{1-\alpha},\quad z \in\bbD, \quad \alpha>0.
\label{LA2}
\end{equation}
The branch of the analytic function $\log ( \zeta - z)$ is fixed by the condition 
$\log ( \zeta - z)=\log ( 1-r)+i\varphi_{0}$ if $z=r \zeta$, $r\in (0,1)$,  and $\zeta =e^{i\varphi_{0}}$. 
We fix $\arg \bigl(-\log ( \zeta - z)+ u (z) \bigr)$ by the condition that it tends to zero as $z=r e^{i\varphi_{0}}$ and $r\to 1-0$. Obviously the function $\omega(z)$ is analytic in the  unit disc $\bbD$ and is smooth up to the boundary $\bbT$, except at the point $z=\zeta$. Let us find its asymptotic behavior as $z\in \bbT$ and $z\to \zeta$.
   
\begin{lemma}\label{Anpr}
Let  $\mu =e^{i \psi}$, $ \zeta  =e^{i \psi_0}$ and $\theta:= \psi  - \psi_0 $. Then the function \eqref{LA2} admits the representation
$$
\omega (\mu)=\bigl(-\log |\theta|+u_\pm (\theta)\bigr)^{1-\alpha} , \quad \pm \theta>0,
$$
where $u_\pm$ are $C^\infty$-smooth functions, 
\begin{equation}
u_{ \pm}(\theta)= \pm i\pi/2 - i\theta/2 - \log \frac{\sin (\theta/2) }{ \theta/2 } +u  (e^{i(\psi_{0}+\theta)}) - i \psi_{0}.
\label{LA5}
\end{equation}
In particular,
\begin{equation}
u_{ +}(0) - u_{ -}(0)=  i\pi .
\label{LA5U}
\end{equation}
\end{lemma}

\begin{proof}
Observe that 
\begin{multline*}
\log(\zeta -\mu)
=
\log(e^{i\psi_0}-e^{i\psi})
=
\log(e^{i\psi_0}(1-e^{i\theta}))
\\
=
\log(1-e^{i\theta})+i\psi_0
=
 \log(2\sin |\theta/2|) + i \arg (1-e^{i\theta})+i\psi_0
\end{multline*}
and $\arg (1-e^{i\theta}) =  (\mp  \pi+\theta)/2$ for $\pm \theta>0$.
Therefore
$$
-\log ( \zeta - \mu) +u (\mu)=- \log|\theta|+ u_{\pm}(\theta)
$$
where $u_{ \pm}(\theta)$ is given by
\eqref{LA5}.
\end{proof}

Below we consider sums of   functions \eqref{LA2} with variable coefficients. Lemma~\ref{Anpr} allows us to apply Theorem~\ref{nonAnalytic} in the special case $v_{1,\pm}(\theta)=0$.

\begin{theorem}\label{Analytic}
Let   $\zeta_1, \zeta_2,\dots,\zeta_L\in\bbT$ be distinct points, and  let functions   $v_{\ell}, u_{\ell}$, $\ell=1,\ldots, L$, be analytic in     $\bbD$ and  $v_{\ell}, u_{\ell}\in C^\infty(\ov{\bbD})$, and assume that \eqref{LA2a} is satisfied for all $u_\ell$, $\zeta_\ell$.  Put
$$
\omega (z)=\sum_{\ell=1}^L v_{\ell} (z)\bigl(-\log ( \zeta_{\ell}- z)+u_{\ell} (z)\bigr)^{1-\alpha},  \quad \alpha>0 .
$$
Then there exists the limit
$$
\lim_{n\to\infty}n^\alpha
\rho_n^+(\omega)
= |1-\alpha| \varkappa(\alpha)\Bigl(\sum_{\ell=1}^L \abs{v_\ell (\zeta_\ell)} ^{1/\alpha}\Bigr)^\alpha
$$
where 
the coefficient 
$\varkappa(\alpha)$ is given by
\eqref{a7}.
\end{theorem} 
\begin{proof} 
It follows from Lemma~\ref{Anpr} that the function $\omega(z)$ admits representation \eqref{eq:ass} where every function 
$\omega_{\ell}(z)$ satisfies Assumption~\ref{AsLog}  with the corresponding functions $v_{1,\pm} (\theta)=0$. Therefore according to relations \eqref{b4ab} and \eqref{LA5U} we have
$b (\overline{\omega_\ell} ) =  (1-\alpha) v_\ell (\zeta_\ell)$.
Now we can apply Theorem~\ref{nonAnalytic}; the smooth error term $\wt\omega$ does not affect the 
asymptotics --- see Remark~\ref{rmk.error}. 
\end{proof}

\section{Fourier transforms of    functions with  logarithmic singularities}

\subsection{Statement of the result}

Our goal in this section is to prove Theorem~\ref{lma.b1}.   
In fact, we prove a slightly more general statement, where Fourier coefficients 
are replaced by Fourier transforms. 
It is convenient to introduce the function of $x\in\bbR$,
\begin{equation}
\Omega(x)=
\begin{cases}
\omega(e^{ix}) & -\pi<x\leq\pi
\\
0 & \text{ otherwise.}
\end{cases}
\label{eq:Omega}
\end{equation}

\begin{theorem}\label{lma.B1C}
Under Assumption~$\ref{AsLog}$, the Fourier transform  $\wh{\Omega} (t)$ of $\Omega (x)$ is a $C^\infty$ function on $\bbR$, which can be written as
\begin{equation}
(2\pi)^{-1/2}
\wh{\Omega}(- t)
= 
{  b}  t^{-1}(\log t)^{-\alpha}
+
G (- t), \quad t>1, 
\label{b5c}
\end{equation}
where $b=b(\omega)$ is given by \eqref{b4ab}
and the error term $G(t)$ satisfies  the estimates 
\begin{equation}
G^{(m)} (- t)=O \bigl(t^{-1-m}(\log t)^{-\alpha-1}\bigr), \quad t \to+\infty,
\label{eq:c8Fc}
\end{equation}
for all $m=0,1,\ldots$. 
\end{theorem}

Theorem~\ref{lma.B1C} will be proven in the rest of this section. Assuming this theorem, 
we can give

\begin{proof}[Proof of Theorem~\ref{lma.b1}] 
Observe that
$$
\wh \omega(- j)
=
\frac1{2\pi}\int_{-\infty}^\infty \Omega(x)e^{ijx}dx= (2\pi)^{-1/2}
\wh{\Omega }(- j).
$$
So   the asymptotics \eqref{b5} for the Fourier coefficients $\wh\omega(-j)$ 
with   the error term $g(-j)= G (-j)$ follows from the asymptotics \eqref{b5c} for the 
Fourier transform $\wh \Omega(-t)$. 
We only have to check that the estimates \eqref{eq:c8Fc}
for the function $G $ and its derivatives yield the estimates \eqref{eq:c8F}
for the sequence $g$. This elementary statement follows, 
for example, from the explicit formula 
$$
g^{(m)}(-j)
=
\int_0^1 dt_1\int_0^1 dt_2\cdots\int_0^1dt_m\,
G^{(m)}(- j+t_1+\dots+t_m),
$$
which can be checked by induction in $m$. 
\end{proof}

 Theorem~\ref{lma.B1C} is proven below through a sequence of steps.
In Lemma~\ref{Lapl} we compute the asymptotics of the Laplace transform
of explicit functions with a logarithmic singularity. In Lemma~\ref{L}, 
we use a contour deformation argument to reduce the question of asymptotics of the 
Fourier transform to that of the Laplace transform.
In Lemma~\ref{lma.d4} we show that the functions $v_{j,\sigma}$ and $u_{j,\sigma}$ 
in the definition of $\omega$ (see Assumption~\ref{AsLog}) can be replaced by their values at zero.
The proof of  Theorem~\ref{lma.B1C} is concluded in Section~4.4.

\subsection{Laplace and Fourier transforms of   logarithmic functions}

Let us start with an elementary result on the  asymptotic expansion of the 
Laplace transform.

\begin{lemma}\label{Lapl}
Let   $\alpha\in {\bbR}$, $m \in\bbZ_{+}$ and
let $c\in(0,1)$.  Then
\begin{multline}
\int_0^{c} \bigl(-\log y\bigr)^{-\alpha} y^m   e^{-yt}dy 
= 
t^{-1-m} ( \log t )^{-\alpha}
\\
\times \bigl(m! +\alpha \Gamma' (m+1) (\log t)^{-1} +O ( (\log t)^{-2}) \bigr)
\label{eq:L2}
\end{multline} 
as $t\to+\infty$ where $\Gamma'$ is the derivative of the Gamma function $\Gamma$.
\end{lemma}

\begin{proof}
First, we split    the integral \eqref{eq:L2} into the integrals
over $(0,t^{-1/2})$ and over $(t^{-1/2}, c)$. 
Due to the factor $e^{-y t}$ the second integral decays faster than any power of $t^{-1}$.
 Making the change of variables $x=y t$, we see that the first integral equals
\begin{equation}
t^{-1-m}(\log t )^{-\alpha}\int_0^{t^{1/2}}  
  \bigl(1- \tfrac{\log x}{\log t}\bigr)^{-\alpha} x^m e^{-x}d x.
\label{eq:Lu}
\end{equation}
Since $u= - \frac{\log x}{\log t} \geq -1/2$ for $x\leq t^{1/2}$, 
we can use the estimate  
\begin{equation}
\Abs{ (1+u)^{-\alpha} -1+\alpha u}
\leq 
C u^2, 
\quad 
u\geq-1/2.
\label{eq:L}
\end{equation}
Thus the integral \eqref{eq:Lu} equals
\begin{equation}
 t^{-1-m}  (\log t )^{-\alpha}
\int_0^{t^{1/2}}  (1+\alpha \tfrac{\log x}{\log t}) \, x^m e^{-x}d x+R(t),
\label{eq:L4q}
\end{equation} 
where the remainder satisfies the estimate 
$$
\abs{R(t)} \leq Ct^{-1-m}(\log t )^{-\alpha-2}\int_0^{\infty}   ( \log x  )^2 \, x^m  e^{-x}d x.
$$
The integral in \eqref{eq:L4q} can be extended to ${\bbR}_{+}$ and then calculated in terms of the Gamma function.  
The arising error decays   faster than any power of $t^{-1}$ as $t\to + \infty$. 
This yields \eqref{eq:L2}.
\end{proof}

The full asymptotic  expansion of the Laplace transform \eqref{eq:L2} is of course well known; 
see,  e.g.  Lemma~3 in \cite{Erdelyi}, but the above proof is slightly simpler than that in \cite{Erdelyi}.   
Here we need two terms, but the method allows one to easily   obtain the full expansion.
 
Next, we discuss the Fourier transform.
Below we suppose that $\arg x>0$ for $x>0$. 
Fix some complex number $a$. We  choose  a number $c > 0$ so small that $ -\log x +a \neq 0$  for $x\in (0,c)$.
 
\begin{lemma}\label{L}
Let $a\in {\bbC}$, $\alpha\in {\bbR}$ and $m \in\bbZ_{+}$.
Suppose that a $C^\infty$ function
$\chi_0$  satisfies condition \eqref{d7} with a sufficiently small $c $. 
Then 
\begin{multline}
\int_{-\infty}^\infty (-\log |x| +a)^{-\alpha} \1_{\pm} (x) \chi_0 (x) x^m e^{ixt}dx
=  
\pm i^{m+1}t^{-m-1}  (\log t)^{-\alpha}
\\
\times \Bigl(m!  + \alpha   \bigl(\Gamma' (m+1) + m! (\pm \pi i/2 - a) \bigr)  (\log t)^{-1}+O( (\log t)^{-2})\Bigr)
\label{eq:yy}
\end{multline}
as $t\to + \infty$. 
\end{lemma}

\begin{proof}
Consider first the sign $``+"$.
For $x\in(0,c )$, denote
\begin{equation}
A(x, t)=\int_0^x \bigl(-\log z + a\bigr)^{-\alpha} z^m e^{ izt}dz.
\label{eq:y2}\end{equation}
Integrating by parts, we see that
\begin{equation}
 \int_0^\infty \bigl(-\log x + a \bigr)^{-\alpha}\chi_0 (x) x^m e^{ ixt}dx=-\int_0^\infty  A(x, t)\chi_0 ' (x)  dx
\label{eq:y3}\end{equation}
where $\chi_0'\in C_{0}^\infty ({\bbR}_{+})$.
Our plan is to find the asymptotics of $A(x, t)$ as $t\to + \infty$ for $x$ in compact subsets of $(0,c )$
and to substitute it into \eqref{eq:y3}.
Let us choose $\kappa>0$ so small that $ -\log z +a \neq 0$ for $z$  in the closed rectangle in the complex plane
with the vertices $0, i \kappa, i \kappa +c , c $.
Instead of $(0,x)$, we can integrate over the line segments $(0,i\kappa)$, $(i\kappa, i \kappa+x)$ and $(i\kappa +x,x)$
 in \eqref{eq:y2}:
\begin{multline*}
A(x,t)
=
\Biggl(\int_0^{i \kappa}+\int_{i \kappa}^{i \kappa+x}+\int_{i \kappa+x}^x\Biggr)
(-\log z + a)^{-\alpha}z^m e^{izt}dz
\\
=:
A_0(t)+A_1(x,t)+A_2(x,t). 
\end{multline*}

Let us first consider the integral over  $(0, i \kappa)$. 
Setting $z= iy$ and using \eqref{eq:L}, we see that
\begin{multline*}
A_{0}( t)
= 
i^{m+1} \int_0^{\kappa} (-\log  y-  i\pi/2 +a )^{-\alpha} y^m e^{-yt}dy
\\
=  
i^{m+1} \int_0^{\kappa}  (-\log y)^{-\alpha}  \bigl(1+ (i\pi/2 - a) \alpha  (-\log y)^{- 1} + \varepsilon(y) \bigr)  y^m e^{-yt}dy 
\end{multline*}
where $\varepsilon(y)= O ( (\log y)^{-2})$  as $y\to 0$. 
Thus we have reduced the question to computing the asymptotics of the Laplace transform. 
Now it follows from formula \eqref{eq:L2} that
\begin{multline*}
A_{0} (t) = i^{m+1}t^{-1-m}   (\log t)^{-\alpha}
\\
\times\bigl(m!  + \alpha (\Gamma' (m+1) + m!(i\pi/2 - a )  )(\log t)^{-1} +O ( (\log t)^{-2}) \bigr).  
\end{multline*} 
Substituting this asymptotics into the integral \eqref{eq:y3}, we get the right-hand side of \eqref{eq:yy}. 

It remains to show that the terms $A_1$ and $A_2$ do not contribute to the 
asymptotics of the integral \eqref{eq:y3}.
Making the change of variables $z= i \kappa +y$, we see that  the integral
 \begin{multline*}
A_{1}(x, t)
=\int_{i \kappa}^{  i \kappa+x } \bigl(-\log z + a\bigr)^{-\alpha} z^m e^{ i z t}d z 
\\
=   e^{-\kappa t} \int_0^{x} \bigl(-\log( i \kappa+y) +a\bigr)^{-\alpha} ( i \kappa +y)^m e^{ iyt}dy
\end{multline*}
decays exponentially as $t \to \infty$. 
This implies that the contribution of  
$A_{1}(x, t)$  to the integral \eqref{eq:y3} also
decays exponentially.

Similarly, making the change of variables $z=x + i  \kappa y $, we can rewrite $A_2$ as
\begin{multline*}
A_{2}(x, t)
=
\int_{ i \kappa+x}^x \bigl(-\log z +a \bigr)^{-\alpha} z^me^{ izt}dz 
\\
=  - i\kappa e^{ itx} \int_{0}^1 \bigl(- \log( x+  i \kappa y) + a\bigr)^{-\alpha}( x + i \kappa y)^m e^{-\kappa yt}dy.
\end{multline*}
In the right-hand side we can integrate by parts arbitrarily many times.
It is important that the integrand $\bigl(- \log( x+  i \kappa y) + a\bigr)^{-\alpha}( x + i \kappa y)^m$ is 
a $C^\infty$-smooth function of $y\in[0,1]$. 
The contribution of the point $y =1$ decays exponentially and 
therefore we have the asymptotic expansion
\begin{equation}
A_{2}(x, t) =   e^{ itx}\sum_{n=1}^N a_{n}(x) t^{-n}+O(t^{-N-1}), \quad t\to + \infty, \quad \forall N>0,
\label{eq:y6}
\end{equation}
with some functions $a_{n}(x)$ that are smooth on the  interval $(0,c )$. 
Substituting \eqref{eq:y6} into \eqref{eq:y3} and integrating by parts with respect to $x$, 
we see that the contribution of   $A_{2}(x, t)$ decays faster than any power of $t^{-1}$ as $t\to\infty$.

To prove  \eqref{eq:yy} for the sign $``-"$, 
we take the relation  \eqref{eq:yy} for the sign $``+"$, 
make the change of the variables $x\mapsto -x$ and pass to the complex conjugation. 
\end{proof}

For $a=0$, the asymptotics of the oscillating  integral \eqref{eq:yy} is  well known (see \cite{WL}), 
although our proof seems to be somewhat simpler than that in \cite{WL} even in this case. 
So we have given the proof of  Lemma~\ref{L}  mainly for the completeness of our presentation. Of course the method of proof of Lemma~\ref{L}   
yields the complete asymptotic expansion of the integral \eqref{eq:yy}, but we do not need it.

\subsection{Replacing $v (x)$ by $v(0)$ and $u (x)$ by $u(0)$}

Here our aim is to prove that the variable parameters $v_{j,\sigma}$ and $u_{j,\sigma}$ in
the definition of the function $\omega$ (see Assumption~\ref{AsLog}) can be replaced by 
their values at zero.

\begin{lemma}\label{lma.d4}
Let  $\alpha\in {\bbR}$, $m\in {\bbZ}_{+}$, $\sigma=\pm$, and let functions $v, u \in C^\infty$.
Then  
 \begin{multline}
\int_{-\infty}^\infty
\Bigl(v  (x)(-\log\abs{x} +  u (x))^{ -\alpha }- v  (0)(-\log\abs{x} +  u (0))^{ -\alpha } \Bigr)\1_\sigma(x) \chi_0(x) x^m e^{ixt}dx
\\
= O (t^{-\rho}  ), \quad \forall \rho <m+2,
\label{eq:repl2}
 \end{multline}
as $t\to\infty$. 
\end{lemma}
\begin{proof}
Let $\varphi \in C^{m+1}$ and $\varphi^ {(m+2)} \in L^{1}_{\text{loc}}$.
Integrating by parts $m+2$ times, we see that
\begin{equation}
\int_{-\infty}^\infty
\varphi (x) \1_\sigma(x) \chi_0(x)   e^{ixt}dx = O (t^{-m-2}) 
\quad \text{ if } \quad  
\varphi (0)= \cdots = \varphi^ {(m+1)}(0)=0.
\label{eq:repl3}
\end{equation}
Let us use the fact that
$
(1+  z)^{ -\alpha}= 1 -\alpha  z+ R  (z)
$
where the remainder 
\begin{equation}
R \in C^\infty
\quad \text{ and } \quad 
R (0)= R' (0)=0. 
\label{eq:rem1}
\end{equation}
Therefore we have
\begin{multline}
 (-\log\abs{x} +  u (x))^{ -\alpha}
=    (-\log\abs{x} +  u (0))^{ -\alpha}\bigl(1+   w (x)\bigr)^{ -\alpha}
\\
=  (-\log\abs{x} +  u (0))^{ -\alpha}\Bigl( 1 -\alpha    w (x)   + R  (w(x))\Bigr)
\label{rem}
\end{multline}
where 
\begin{equation}
 w (x)= (-\log\abs{x} +  u (0))^{ -1} ( u (x)- u (0)). 
 \label{eq:rem2}
\end{equation}
We substitute  the expression \eqref{rem}  into the integral \eqref{eq:repl2} and consider every term separately. 

First we consider 
$$ v(x) (-\log\abs{x} +  u (0))^{ -\alpha } x^m = ( v(0) + v'(0)x + v_{1} (x) )(-\log\abs{x} +  u (0))^{ -\alpha } x^m$$
where $v_{1}  \in C^\infty$ and $v_{1} (0)= v_{1}' (0)=0$. 
The term with $v(0)$ is cancelled out by the second term in the integrand in \eqref{eq:repl2}.  According to Lemma~\ref{L},   the contribution of $ v'(0)x   $ is bounded by  $C t^{-m-2} (\log t)^{ -\alpha }$, and according to \eqref{eq:repl3}  the   contribution of $ v_{1} (x)$ is bounded by  $C t^{-m-2}  $. 

Next, we consider   the term 
$$
 v(x) (-\log\abs{x} +  u (0))^{ -\alpha-1}    (u(x)- u(0))x^m.
$$
We have $ v(x)     (u(x)- u(0))x^m= v(0) u'(0)x^{m+1} + R_{1} (x)$ where $R_{1} \in C^\infty$ and $R_{1} (0)=\cdots= R_{1}^{(m+1)} (0)=0$.  As we have already seen,  by Lemma~\ref{L}  the contribution of $ v(0) u'(0)x^{m+1}  $ is bounded by  $C t^{-m-2} (\log t)^{ -\alpha -1}$,  and by \eqref{eq:repl3}  the   contribution of $ R_{1} (x)$ is bounded by  $C t^{-m-2}  $.

 It remains to consider the function 
\begin{equation}
\varphi(x)= v(x) (-\log\abs{x} +  u (0))^{ -\alpha}  R  (w(x))x^m.
\label{eq:rem3}
\end{equation}
Clearly, $w\in C^\infty ((-c, c ) \setminus \{ 0\})$ and differentiating \eqref{eq:rem2} we see that $w^{(k)}(x)=O(|x|^{1-k})$ as $x\to 0$ for all $k=0,1,\ldots$. Therefore differentiating the composite function $R  (w(x))$ and taking into account \eqref{eq:rem1}, we find that
$$
\frac{d^k}{dx^k}  R  (w(x)) = O(|x|^{2-k}),\quad k=0,1,\ldots.
$$
Finally,  differentiating the product \eqref{eq:rem3}, we see that $\varphi \in C^\infty ((-c , c ) \setminus \{ 0\})$ and 
$$
  \varphi^{(k)}(x) = O((-\log\abs{x}  )^{ -\alpha}  |x|^{2+m-k}), \quad k=0,1,\ldots,
$$
as $x\to 0$. Thus we can   apply relation \eqref{eq:repl3} to the function \eqref{eq:rem3}.
\end{proof}

Putting together Lemmas~\ref{L} and \ref{lma.d4}, we obtain the following result.

\begin{lemma}\label{repl}
Let the assumptions of Lemma~\ref{lma.d4} be satisfied. Then, as $t\to+\infty$,
\begin{multline}
 \int_{-\infty}^\infty v(x) (-\log |x |+ u(x))^{-\alpha} \1_{\pm} (x) \chi_0 (x)  x^m e^{ixt}dx
=  
\pm  i^{m+1} t^{-m-1}  (\log t)^{-\alpha}
\\
\times v(0) \Bigl(m!  + \alpha   \bigl(\Gamma' (m+1) + m! (\pm \pi i/2 - u(0)) \bigr)  (\log t)^{-1}+O( (\log t)^{-2})\Bigr).
\label{eq:yyP}
\end{multline}
\end{lemma}

\subsection{Proof of Theorem~\ref{lma.B1C}}

Since
$$
\biggl(\frac{d}{dt}\biggr)^m \bigl( t^{-1}(\log t)^{-\alpha} \bigr)
 =  
 (- 1)^{m}
m! t^{-1-m}(\log t)^{-\alpha} \bigl(1+ O ( (\log t)^{-1})\bigr),  
$$
we need to prove that 
\begin{equation}
\int_{-\infty}^\infty \Omega(x) x^m   e^{ixt}dx
=
2\pi b i^m m! t^{-1-m}(\log t)^{-\alpha}  \bigl(1+ O ( (\log t)^{-1})\bigr)
\label{*0}
\end{equation}
as $t\to+\infty$, where the asymptotic coefficient $b  $ is given by equality  \eqref{b4ab}. 

Recall that the function $\Omega(x)  $ is defined by formulas \eqref{b4cx} and \eqref{eq:Omega}. So we only have to substitute this expression for $\Omega(x)$ into the left-hand side and to use Lemma~\ref{repl}. 
Thus, keeping only the leading term in  the right-hand side of \eqref{eq:yyP}, we find that
\begin{multline}
\sum_{\sigma=\pm} \int_{-\infty}^\infty v_{1, \sigma} (x)(-\log\abs{x} +  u_{1, \sigma }(x))^{-\alpha }\1_\sigma(x) \chi_0(x)x^m   e^{ixt}dx
\\
=
  i^{m+1} m! (v_{1, +} (0) - v_{1, -} (0) ) t^{-m-1}(\log t)^{-\alpha}
\bigl( 1 + (\log t)^{ -1}\bigr) .
\label{eq:Z1}
\end{multline}
Similarly, using \eqref{eq:yyP} with $\alpha$ replaced by $\alpha-1$, we obtain
\begin{multline}
\sum_{\sigma=\pm} \int_{-\infty}^\infty v_{0, \sigma} (x)(-\log\abs{x} +  u_{0, \sigma }(x))^{1-\alpha }\1_\sigma(x) \chi_0(x)x^m   e^{ixt}dx =   i^{m+1} t^{-m-1}  (\log t)^{-\alpha}
\\
\times
\biggl\{v_{0,+}(0) \biggl(m! \log t + (\alpha  -1) \bigl(\Gamma' (m+1) + m! ( \pi i/2 - u_{0,+}(0)) \bigr) \biggr)
\\
- v_{0,-}(0) \Bigl(m! \log t + (\alpha  -1) \bigl(\Gamma' (m+1) + m! (- \pi i/2 - u_{0,-}(0)) \bigr) \Bigr)
+O( (\log t)^{-1}) \biggr\}.
\label{eq:Z2}
\end{multline}
Taking into account condition \eqref{eq:bbx}, we see that the right-hand side here equals
$$
 i^{m+1} m!  t^{-m-1}  (\log t)^{-\alpha} v_{0} (\alpha  -1)  ( \pi i - u_{0,+}(0)+ u_{0,-}(0)) \bigl( 1+O( (\log t)^{-1})\bigr).
 $$
 Thus putting together relations  \eqref{eq:Z1}, \eqref{eq:Z2}, we conclude the proof of \eqref{*0} and hence of Theorem~\ref{lma.B1C}.

\appendix
\section{Besov and Besov-Lorentz spaces}

Here for completeness we recall the definitions of the Besov  class  $B_{p,p}^{1/p}$ and 
the Besov-Lorentz  class $\mathfrak B_{p,\infty}^{1/p}$  on $\bbT$. 
The parameter  $p>0$ is arbitrary. 
We refer to     the books \cite{Peller} (see Section 6.4 and Appendix 2) and \cite{Triebel}  for more details.

Let $w\in C_0^\infty(\bbR)$ be a function with the properties $w\geq0$, $\supp w=[1/2,2]$
and 
$$
\sum_{n=0}^\infty w(t/2^n)=1, \quad \forall t\geq1.
$$
Let   $f$ be a distribution on $L^1 (\bbT)$ with  the   Fourier coefficients $\wh f(j)$, ${j\in\bbZ}$. 
For $n\in \bbZ$, let us denote by $f_n$ the polynomial 
$$
f_n(\mu)=\sum_{j\in\bbZ} w(\pm j/2^{|n|})\wh f(j)\mu^j, \quad \mu\in \bbT, \quad \pm n >0,
$$
and let
$f_0(\mu)=\wh f(1)\mu+ \wh f(0) + \wh f(-1)\overline{\mu}$.
The Besov class $B_{p,p}^{1/p}$ is defined by the condition
\begin{equation}
\sum_{n\in\bbZ} 2^{\abs{n}}\norm{f_n}_{L^p}^p <\infty. 
\label{besov}
\end{equation}

By definition, 
$f\in \mathfrak B_{p,\infty}^{1/p}$ if and only if
$$
\sup_{t>0}t^p \sum_{n\in\bbZ} 2^{\abs{n}} m (\{\mu\in\bbT: \abs{f_n(\mu)}>t\})<\infty
$$
which is the ``weak version" of the condition \eqref{besov}. 
We have 
$$
B^{1/ p}_{p,p} \subset \mathfrak{B}^{1/ p}_{p,\infty}\subset B^{1/ q}_{q,q}, \quad \forall q>p.
$$

The H\"older-Zygmund class $\Lambda_\alpha$, $\alpha>0$, is defined in terms of the difference operator
$$
(\Delta_{\tau}f) (\mu)= f(\tau \mu)- f(\mu), \quad \tau \in\bbT.
$$
By definition, 
$f\in \Lambda_\alpha$ if and only if
$$
\|(\Delta_{\tau}^n f )(\mu)\|_{L^\infty}\leq C |\tau -1|^\alpha
$$
where $n$ is an arbitrary integer such that $n>\alpha$.   Observe that $\Lambda_\alpha$ coincides with the H\"older class $C^\alpha$ if $\alpha$ is not integer and $C^\alpha\subset \Lambda_\alpha$ if $\alpha$  is an integer. We also note that
$\Lambda_{\alpha}\subset \mathfrak{B}^{\alpha}_{1/\alpha,\infty}$.

\section*{Acknowledgements}
The authors are grateful to the Departments of Mathematics of King's College London and of the University of Rennes 1 (France)  for the financial support. The second author (D.Y.) acknowledges also the  support and hospitality of the Isaac Newton Institute for Mathematical Sciences (Cambridge  University, UK) where a part of this work has been done during the program Periodic and Ergodic Spectral Problems.

\end{document}